\documentclass[11pt,leqno]{article}
\usepackage{hyperref}
\usepackage{soul}
\usepackage{pdfsync}
\usepackage[doc]{optional}
\usepackage{xcolor}
\definecolor{labelkey}{rgb}{0,0.08,0.45}
\definecolor{rekey}{rgb}{0,0.6,0.0}
\definecolor{Brown}{rgb}{0.45,0.0,0.05}
\usepackage{exscale,relsize}
\usepackage{amsmath}
\usepackage{amsfonts}
\usepackage{amssymb}
\usepackage{calc}
\usepackage{theorem}
\usepackage{pifont}      
\usepackage{graphicx}
\usepackage[toc,page]{appendix}
\usepackage{appendix}
\DeclareMathOperator{\weakstarly}{\rightharpoondown_{\mathrm{w*}}}

\newcommand{\scal}[2]{\langle{{#1},{#2}}\rangle}

\newcommand{\RR}{\ensuremath{\mathbb R}}

\newcommand{\RX}{\ensuremath{\,\left]-\infty,+\infty\right]}}

\newcommand{\NN}{\ensuremath{\mathbb N}}

\newcommand{\menge}[2]{\big\{{#1} \mid {#2}\big\}}

\newcommand{\To}{\ensuremath{\rightrightarrows}}

\newcommand{\spand}{\operatorname{span}}

\newcommand{\sta}{\operatorname{star}}

\newcommand{\aff}{\operatorname{aff}}

\newcommand{\dom}{\ensuremath{\operatorname{dom}}}

\newcommand{\gra}{\ensuremath{\operatorname{gra}}}

\newcommand{\intdom}{\ensuremath{\operatorname{int}\operatorname{dom}}\,}
\newcommand{\inte}{\ensuremath{\operatorname{int}}}

\newcommand{\bd}{\ensuremath{\operatorname{bdry}}}

\renewcommand{\phi}{\ensuremath{\varphi}}

\newcommand{\supp}{\ensuremath{\operatorname{supp}}}

\newcommand{\qede}{\hspace*{\fill}$\Diamond$\medskip}

\newtheorem{theorem}{Theorem}[section]
\newtheorem{lemma}[theorem]{Lemma}
\newtheorem{fact}[theorem]{Fact}
\newtheorem{corollary}[theorem]{Corollary}
\newtheorem{proposition}[theorem]{Proposition}

\theoremstyle{plain}{\theorembodyfont{\rmfamily}
}
\theoremstyle{plain}{\theorembodyfont{\rmfamily}
}
\theoremstyle{plain}{\theorembodyfont{\rmfamily}
}
\theoremstyle{plain}{\theorembodyfont{\rmfamily}
\newtheorem{example}[theorem]{Example}}
\theoremstyle{plain}{\theorembodyfont{\rmfamily}
\newtheorem{remark}[theorem]{Remark}}
\newtheorem{problem}[theorem]{Open Problem}
\theoremstyle{plain}{\theorembodyfont{\rmfamily}
}

\begin{document}


\title{\sffamily{
Sum theorems  for maximally monotone operators of type (FPV)}}

\author{
Jonathan M. Borwein\thanks{CARMA, University of Newcastle,
Newcastle, New South Wales 2308, Australia. E-mail:
\texttt{jonathan.borwein@newcastle.edu.au}. Laureate Professor at
the University of Newcastle and Distinguished Professor at  King
Abdul-Aziz University, Jeddah.}\; and Liangjin\ Yao\thanks{CARMA,
University of Newcastle,
 Newcastle, New South Wales 2308, Australia.
E-mail:  \texttt{liangjin.yao@newcastle.edu.au}.}}

\date{May 29,  2013}
\maketitle

\begin{abstract} \noindent
The most important open problem in Monotone Operator Theory
concerns the maximal monotonicity of the sum of two
maximally monotone operators provided that the classical
Rockafellar's constraint qualification holds.

In this paper, we establish the maximal monotonicity of $A+B$
provided that
$A$ and $B$ are maximally monotone operators such that
$\sta(\dom A)\cap\inte\dom B\neq\varnothing$, and
$A$ is of type (FPV). We  show that when also $\dom A$ is convex, the sum operator: $A+B$ is also of type (FPV).
Our  result generalizes and unifies several recent sum theorems.
\end{abstract}



\section{Introduction}

Throughout this paper, we assume that
$X$ is a real Banach space with norm $\|\cdot\|$,
that $X^*$ is the continuous dual of $X$, and
that $X$ and $X^*$ are paired by $\scal{\cdot}{\cdot}$.
Let $A\colon X\To X^*$
be a \emph{set-valued operator}  (also known as a relation, point-to-set mapping or multifunction)
from $X$ to $X^*$, i.e., for every $x\in X$, $Ax\subseteq X^*$,
and let
$\gra A := \menge{(x,x^*)\in X\times X^*}{x^*\in Ax}$ be
the \emph{graph} of $A$.
Recall that $A$ is  \emph{monotone} if
\begin{equation*}
\scal{x-y}{x^*-y^*}\geq 0,\quad \forall (x,x^*)\in \gra A\;\forall (y,y^*)\in\gra A,
\end{equation*}
and \emph{maximally monotone} if $A$ is monotone and $A$ has no proper monotone extension
(in the sense of graph inclusion).
Let $A:X\rightrightarrows X^*$ be monotone and $(x,x^*)\in X\times X^*$.
 We say $(x,x^*)$ is \emph{monotonically related to}
$\gra A$ if
\begin{align*}
\langle x-y,x^*-y^*\rangle\geq0,\quad \forall (y,y^*)\in\gra A.\end{align*}
Let $A:X\rightrightarrows X^*$ be maximally monotone. We say $A$ is
\emph{of type (FPV)} if  for every open convex set $U\subseteq X$ such that
$U\cap \dom A\neq\varnothing$, the implication
\begin{equation*}
x\in U\,\text{and}\,(x,x^*)\,\text{is monotonically related to $\gra A\cap (U\times X^*)$}
\Rightarrow (x,x^*)\in\gra A
\end{equation*}
holds. We emphasize that it remains possible that all maximally monotone operators are of type (FPV). Also  every (FPV) operator has the closure of its domain convex. See \cite{Si2, Si, BorVan,BY5} for this and more information on operators of type (FPV).
We say $A$ is a \emph{linear relation} if $\gra A$ is a
linear subspace.

 Monotone operators have proven  important
 in modern Optimization and Analysis; see, e.g., the books
\cite{BC2011,BorVan,BurIus,ButIus,ph,Si,Si2,RockWets,Zalinescu,Zeidler2A,Zeidler2B}
and the references therein. We adopt standard notation used in these
books: thus, $\dom A:= \menge{x\in X}{Ax\neq\varnothing}$ is the
\emph{domain} of $A$. Given a subset $C$ of $X$, $\inte C$ is the
\emph{interior} of $C$, $\bd{C}$ is the \emph{boundary} of $C$,
$\aff C$ is the \emph{affine hull} of $C$,  $\overline{C}$ is the
norm \emph{closure} of $C$, and $\spand C$
is the span
(the set of all finite linear combinations) of $C$.
The  \emph{intrinsic core} or \emph{relative algebraic interior} of $C$, $^{i}C$ \cite{Zalinescu},
is  defined by
$^{i}C:=\{a\in C\mid \forall x\in \aff(C-C),
\exists\delta>0, \forall\lambda\in\left[0,\delta\right]:
a+\lambda x\in C\}$. We then define $^{ic}C$ by
\begin{equation*}
^{ic}C:=\begin{cases}^{i}C,\,&\text{if $\aff C$ is closed};\\
\varnothing,\,&\text{otherwise},
\end{cases}\end{equation*}

The \emph{indicator function} of $C$, written as $\iota_C$, is defined
at $x\in X$ by
\begin{align*}
\iota_C (x):=\begin{cases}0,\,&\text{if $x\in C$;}\\
\infty,\,&\text{otherwise}.\end{cases}\end{align*}
If $C,D\subseteq X$, we set $C-D=\{x-y\mid x\in C, y\in D\}$.
  For every $x\in X$, the \emph{normal cone} operator of $C$ at $x$
is defined by $N_C(x):= \menge{x^*\in
X^*}{\sup_{c\in C}\scal{c-x}{x^*}\leq 0}$, if $x\in C$; and $N_C(x)=\varnothing$,
if $x\notin C$. We define the \emph{support points} of $C$, written as $\supp C$,
by $\supp C:=\{c\in C\mid N_C (c)\neq\{0\}\}$.
For $x,y\in X$, we set $\left[x,y\right]:=\{tx+(1-t)y\mid 0\leq t\leq 1\}$.
We define the \emph{centre} or \emph{star} of $C$ by $\sta C:=\{x\in C\mid \left[x,c\right]\subseteq C,\,
\forall c\in C\}$ \cite{BorLew}. Then $C$ is convex if and only if $\sta C=C$.

 Given $f\colon X\to \RX$, we set
$\dom f:= f^{-1}(\RR)$.
 We say $f$ is \emph{proper} if $\dom f\neq\varnothing$.
 We also set $P_X: X\times X^*\rightarrow X\colon
(x,x^*)\mapsto x$. Finally,  the \emph{open unit ball} in $X$ is
denoted by $U_X:= \menge{x\in X}{\|x\|< 1}$, the \emph{closed unit
ball} in $X$ is denoted by $B_X:= \menge{x\in X}{\|x\|\leq 1}$, and
$\NN:=\{1,2,3,\ldots\}$. We denote by $\longrightarrow$ and
$\weakstarly$ the norm convergence and weak$^*$ convergence of
nets,  respectively.

Let $A$ and $B$ be maximally monotone operators from $X$ to
$X^*$.
Clearly, the \emph{sum operator} $A+B\colon X\To X^*\colon x\mapsto
Ax+Bx: = \menge{a^*+b^*}{a^*\in Ax\;\text{and}\;b^*\in Bx}$
is monotone.
Rockafellar established the following very important result in 1970.
\begin{theorem}[Rockafellar's sum theorem]
\emph{(See  \cite[Theorem~1]{Rock70} or \cite{BorVan}.)} Suppose
that $X$ is reflexive. Let $A, B: X\rightrightarrows  X^*$ be
maximally monotone. Assume that $A$ and  $B$  satisfy the classical
\emph{constraint qualification}\[\dom A \cap\intdom B\neq
\varnothing.\] Then $A+B$ is maximally monotone.
\end{theorem}

Arguably, the most significant open problem in the theory concerns  the
maximal monotonicity of the sum of two maximally monotone operators
in general Banach spaces; this is called the ``sum problem''. Some
recent developments on the sum problem can be found in  Simons'
monograph \cite{Si2} and \cite{Bor1,Bor2,Bor3,BorVan,BY3, BY2, ZalVoi,
MarSva5,VV2, Yao3,Yao2,YaoPhD}. It is known, among other things, that
the sum theorem holds under Rockafellar's constraint qualification
when both operators are of dense type or when each operator has
nonempty domain interior \cite[Ch. 8]{BorVan} and \cite{ZalVoi2}.

Here we focus on the  case when $A$ is of type (FPV), and $B$ is maximally monotone
 such that
\[\sta(\dom A)\cap\inte\dom B\neq\varnothing.\] (Implicitly this means that $B$ is also of type (FPV).) In Theorem \ref{TePGV:1} we
shall  show that $A+B$ is maximally monotone. As noted  it seems possible  that all maximally monotone operators are of type (FPV).

The remainder of this paper is organized as follows. In
Section~\ref{s:aux}, we collect auxiliary results for future
reference and for the reader's convenience. In Section~\ref{s:main}, our main
result (Theorem~\ref{TePGV:1}) is presented. In Section~\ref{s:cor}, we then provide various corollaries and examples. We also pose several significant open questions
on the sum problem. We  leave the details of proof of Case 2 of Theorem~\ref{TePGV:1} to Appendix~\ref{Appenc:1}.

\section{Auxiliary Results}
\label{s:aux}

We first introduce one of Rockafellar's results.
\begin{fact}[Rockafellar]
\emph{(See \cite[Theorem~1]{Rock69} or
\cite[Theorem~27.1 and Theorem~27.3]{Si2}.)}
\label{f:referee02c}
Let $A:X\To X^*$ be  maximally monotone
 with $\inte\dom A\neq\varnothing$. Then
$\inte\dom A=\inte\overline{\dom A}$
and  $\inte\dom A$ and $\overline{\dom A}$ are both convex.
\end{fact}

The Fitzpatrick function defined below has proven to be an important tool in Monotone
Operator Theory.
\begin{fact}[Fitzpatrick]
\emph{(See {\cite[Corollary~3.9]{Fitz88}}.)}
\label{f:Fitz}
Let $A\colon X\To X^*$ be  monotone,  and set
\begin{equation}\label{ff:def}
F_A\colon X\times X^*\to\RX\colon
(x,x^*)\mapsto \sup_{(a,a^*)\in\gra A}
\big(\scal{x}{a^*}+\scal{a}{x^*}-\scal{a}{a^*}\big),
\end{equation}
 the \emph{Fitzpatrick function} associated with $A$.
Suppose also $A$ is maximally monotone. Then for every $(x,x^*)\in X\times X^*$, the inequality
$\scal{x}{x^*}\leq F_A(x,x^*)$ is true,
and the equality holds if and only if $(x,x^*)\in\gra A$.
\end{fact}

The next result is central to our arguments.

\begin{fact}
\emph{(See \cite[Theorem~3.4 and Corollary~5.6]{Voi1},
or \cite[Theorem~24.1(b)]{Si2}.)}
\label{f:referee1}
Let $A, B:X\To X^*$ be maximally monotone operators. Assume
$\bigcup_{\lambda>0} \lambda\left[P_X(\dom F_A)-P_X(\dom F_B)\right]$
is a closed subspace.
If
\begin{equation*}
F_{A+B}\geq\langle \cdot,\,\cdot\rangle\;\text{on \; $X\times X^*$},
\end{equation*}
then $A+B$ is maximally monotone.
\end{fact}

We next cite several results regarding operators of type (FPV).

\begin{fact}[Simons]
\emph{(See \cite[Theorem~46.1]{Si2}.)}
\label{f:referee01}
Let $A:X\To X^*$ be a maximally monotone linear relation.
Then $A$ is of type (FPV).
\end{fact}

The following result presents a sufficient condition for
a maximally monotone operator to be of type (FPV).
\begin{fact}[Simons and Verona-Verona]
\emph{(See \cite[Theorem~44.1]{Si2}, \cite{VV1} or \cite{Bor2}.)}
\label{f:referee02a}
Let $A:X\To X^*$ be maximally monotone. Suppose that
for every closed convex subset $C$ of $X$
with $\dom A \cap \inte C\neq \varnothing$, the operator
$A+N_C$ is maximally monotone.
Then $A$ is of type  (FPV).
\end{fact}

\begin{fact}\label{Ll:l1}\emph{(See \cite[Lemma~2.5]{BWY4}.)}
Let $C$ be  a nonempty closed convex
subset of $X$ such that $\inte C\neq \varnothing$.
Let $c_0\in \inte C$ and suppose that $z\in X\smallsetminus C$.
Then there exists
$\lambda\in\left]0,1\right[$ such
that $\lambda c_0+(1-\lambda)z\in\bd C$.
\end{fact}

\begin{fact}[Boundedness below]\emph{(See  \cite[Fact~4.1]{BY1}.)}\label{extlem}
Let $A:X\rightrightarrows X^*$ be monotone and $x\in\inte\dom A$.
 Then there exist $\delta>0$ and  $M>0$ such that
 $x+\delta B_X\subseteq\dom A$ and
 $\sup_{a\in x+\delta B_X}\|Aa\|\leq M$.
Assume that $(z,z^*)$ is monotonically related to $\gra A$. Then
\begin{align*}
\langle z-x, z^*\rangle
\geq \delta\|z^*\|-(\|z-x\|+\delta) M.
\end{align*}
\end{fact}

\begin{fact}[Voisei and Z\u{a}linescu]
\emph{(See  \cite[Corollary~4]{ZalVoi}.)}
\label{voiZalsm}
Let $A, B:X\To X^*$ be  maximally monotone.  Assume that $^{ic}(\dom A)\neq\varnothing,
^{ic}(\dom B)\neq\varnothing$ and $0\in^{ic}\left[\dom A-\dom B\right]$.
Then $A+B$ is maximally monotone.
\end{fact}

The proof of the next Lemma~\ref{rcf:01}  follows closely the lines of
 that of \cite[Lemma~2.10]{BY3}. It generalizes
  both \cite[Lemma~2.10]{BY3} and \cite[Lemma~2.10]{BWY9}.
\begin{lemma}\label{rcf:01}
Let $A:X\To X^*$ be   monotone, and let
$B:X\rightrightarrows X^*$ be a maximally monotone operator. Suppose
that $\sta(\dom A) \cap \inte \dom B\neq \varnothing$. Suppose also that
$(z,z^*)\in X\times X^*$ with $z\in\dom A$ is monotonically related to $\gra (A+ B)$.
 Then  $z\in  \dom B$.

\end{lemma}
\begin{proof}
We can and do suppose that $(0,0)\in\gra A\cap\gra B$ and $0\in\sta(\dom A) \cap \inte \dom B$.
Suppose to the contrary that $z\notin\dom B$. Then we have $z\neq0$.
We claim that
\begin{align}
N_{\left[0,z\right]}+B \quad\text{is maximally monotone}.\label{Lrtg:1}
\end{align}
Because $z\neq0$,  we have $\tfrac{1}{2}z\in{^{ic}(\dom N_{\left[0,z\right]})}$.
 Clearly,
$^{ic}(\dom B)\neq\varnothing$ and $0\in{^{ic}
\left[\dom N_{\left[0,z\right]}-\dom B\right]}$.
By Fact~\ref{voiZalsm},
$N_{\left[0,z\right]}+B$ is maximally monotone and hence \eqref{Lrtg:1} holds.
Since $(z, z^*)\notin\gra (N_{\left[0,z\right]}+B)$,
there exist $\lambda \in\left[0,1\right]$
and $x^*, y^*\in X^*$ such that $(\lambda z, x^*)\in\gra N_{\left[0,z\right]}$,
$(\lambda z, y^*)\in\gra B$ and
\begin{align}
\langle z-\lambda z, z^*-x^*-y^*\rangle<0.\label{Lrtg:2}
\end{align}

Now $\lambda<1$, since $(\lambda z, y^*)\in\gra B$ and $z\notin\dom B$,  by
\eqref{Lrtg:2},
\begin{align}
\langle  z, -x^*\rangle+\langle  z, z^*-y^*\rangle=\langle  z, z^*-x^*-y^*\rangle<0.
\label{Lrtg:3}
\end{align}
Since $(\lambda z, x^*)\in\gra N_{\left[0,z\right]}$,
we have $\langle z-\lambda z, x^*\rangle\leq0$. Then  $\langle z, -x^*\rangle\geq0$. Thus
\eqref{Lrtg:3} implies that
\begin{align}
\langle  z, z^*-y^*\rangle<0.
\label{Lrtg:4}
\end{align}
Since $0\in\sta(\dom A)$ and $z\in\dom A$,
$\lambda z\in\dom A$.
By the assumption on $(z,z^*)$, we have
\begin{align*}
\langle  z-\lambda z, z^*- a^*-y^*\rangle\geq0,
\quad\forall a^*\in A(\lambda z).
\end{align*}
Thence, $\langle z, z^*- a^*-y^*\rangle\geq0$
 and hence
\begin{align}
\langle z, z^*-y^*\rangle\geq\langle z, a^*\rangle,
\quad\forall a^*\in A(\lambda z).\label{Lrtg:5}
\end{align}
Next we show that
\begin{align}
\langle z, a^*\rangle\geq0,\quad \exists a^*\in A(\lambda z).\label{Lrtg:6}
\end{align}
We consider two cases.

\emph{Case 1}: $\lambda=0$. Then take $a^*=0$ to see that \eqref{Lrtg:6} holds.

\emph{Case 2}: $\lambda\neq0$. Let $a^*\in A(\lambda z)$.
 Since $(\lambda z, a^*)\in \gra A$,
$\langle \lambda z, a^*\rangle=\langle \lambda z-0, a^*-0\rangle\geq0$
and hence $\langle  z,
a^*\rangle\geq0$. Hence \eqref{Lrtg:6} holds.

Combining \eqref{Lrtg:5} and \eqref{Lrtg:6},
\begin{align*}
\langle z, z^*-y^*\rangle\geq0,\quad\text{which contradicts \eqref{Lrtg:4}}.
\end{align*}
Hence $z\in\dom B$.
\end{proof}

The proof of Lemma~\ref{LeWExc:1} is modelled on
that of \cite[Proposition~3.1]{Yao3}. It is the first in a sequence of lemmas we give that will allow us to apply Fact \ref{f:referee1}.

\begin{lemma}\label{LeWExc:1}
Let $A:X\To X^*$ be   monotone, and let
$B:X\rightrightarrows X^*$ be maximally monotone. Let $(z, z^*)\in X\times X^*$. Suppose
 $x_0\in\dom A \cap \inte \dom B$ and that there exists a sequence $(a_n, a^*_n)_{n\in\NN}$ in $\gra A\cap\Big(\dom B\times X^*\Big)$ such that $(a_n)_{n\in\NN}$ converges to a point in
$\left[x_0,z\right[$, while
\begin{align}
\langle z-a_n, a^*_n\rangle\longrightarrow+\infty.\label{LeWExc:1E1}
\end{align}
Then
 $F_{A+B}(z,z^*)=+\infty$.

\end{lemma}

\begin{proof}
Since $a_n\in\dom B$ for every $n\in\NN$,
we may pick $v^*_n\in B(a_n)$.
We again consider two cases.

\emph{Case 1}:  $(v^*_n)_{n\in\NN}$ is bounded.

Then we  have
\begin{align*}
F_{A+B}(z,z^*)&\geq \sup_{\{n\in\NN\}}\left[\langle a_n,z^*\rangle+ \langle z-a_n, a_n^*\rangle
+\langle z-a_n,v_n^*\rangle \right]\\
&\geq \sup_{\{n\in\NN\}}\left[-\| a_n\|\cdot\|z^*\|
+ \langle z-a_n, a_n^*\rangle
-\|z-a_n\|\cdot \|v_n^*\| \right]\\
&=+\infty\quad\text{(by \eqref{LeWExc:1E1} and
 the boundedness of $(v^*_n)_{n\in\NN}$)}.
\end{align*}
Hence $F_{A+B}(z,z^*)=+\infty$.

\emph{Case 2}: $(v^*_n)_{n\in\NN}$ is unbounded.

By  assumption, there exists $0\leq\lambda<1$ such that \begin{align}
a_n\longrightarrow x_0+\lambda(z-x_0).\label{LeWExc:1E2}
\end{align}
We first show that
\begin{align}
\limsup_{n\rightarrow\infty}\,\langle z-a_n, v_n^*\rangle=+\infty.
\label{FCG:2}
\end{align}
Since $(v^*_n)_{n\in\NN}$ is unbounded and,
after passing to a subsequence if necessary, we may assume that
$\|v^*_n\|\neq 0,\forall n\in\NN$ and that $\|v^*_n\|\rightarrow +\infty$.
By $x_0\in\inte\dom B$ and Fact~\ref{extlem},
there exist $\delta_0>0$ and $K_0>0$ such that
\begin{align}
\langle a_n-x_0, v^*_n\rangle\geq \delta_0\|v^*_n\|-(\|a_n-x_0\|+\delta_0)K_0.\label{RVT:10a}
\end{align}
Then we have
\begin{align}
 \langle a_n-x_0, \frac{v^*_n}{\|v^*_n\|}\rangle
\geq\delta_0-\frac{(\|a_n-x_0\|+\delta_0) K_0}{\|v^*_n\|},
 \quad \forall n\in\NN.\label{FCG:3}
\end{align}
By the Banach-Alaoglu Theorem
(see \cite[Theorem~3.15]{Rudin}),
 there  exist a weak* convergent \emph{subnet}
$(\frac{v^*_\gamma}{\|v^*_\gamma\|})_{\gamma\in\Gamma}$ of $(\frac{v^*_n}{\|v^*_n\|})_{n\in\NN}$
such that
\begin{align}\frac{v^*_\gamma}{\|v^*_\gamma\|}\weakstarly
 v^*_{\infty}\in X^*.\label{FCGG:9}\end{align}
Using \eqref{LeWExc:1E2} and
taking the limit in \eqref{FCG:3} along the subnet, we obtain
\begin{align}
 \langle \lambda (z-x_0),  v^*_{\infty}\rangle
\geq  \delta_0.
\end{align}
Hence $\lambda$ is strictly positive and
\begin{align}
 \langle  z-x_0,  v^*_{\infty}\rangle
\geq  \tfrac{\delta_0}{\lambda}>0.\label{FCG:03}
\end{align}
Now  assume contrary to \eqref{FCG:2} that there exists $M>0$ such that
\begin{align*}
\limsup_{n\rightarrow\infty}\langle z-a_n, v_n^*\rangle
< M.
\end{align*}

Then, for all $n$ sufficiently large,
\begin{align*}
\langle z-a_n, v^*_n\rangle< M+1,
\end{align*}
and so
\begin{align}
\langle z-a_n, \frac{v^*_n}{\|v^*_n\|}\rangle<
\frac{M+1}{\|v^*_n\|}\label{Rev:a}.
\end{align}
Then by \eqref{LeWExc:1E2} and \eqref{FCGG:9},
taking  the limit in \eqref{Rev:a} along the subnet again, we see that
\begin{align*}
(1-\lambda)\langle  z-x_0,  v^*_{\infty}\rangle\leq 0.
\end{align*}
Since $\lambda<1$, we see $\langle z-x_0,  v^*_{\infty}\rangle\leq 0$
 contradicting \eqref{FCG:03},
and\eqref{FCG:2} holds.
By \eqref{LeWExc:1E1} and \eqref{FCG:2},
\begin{align*}
F_{A+B}(z,z^*)&\geq \sup_{{n\in\NN}}
\left[\langle a_n,z^*\rangle+\langle z-a_n,a_n\rangle
+\langle z-a_n, v_n^*\rangle \right]=+\infty.
\end{align*}
Hence
\begin{align*}
F_{A+B}(z,z^*)=+\infty,
\end{align*}
as asserted.
\end{proof}

We also need the following two  lemmas.

\begin{lemma}\label{LeWExc:2}
Let $A:X\To X^*$ be  monotone, and let
$B:X\rightrightarrows X^*$ be maximally monotone. Let $(z, z^*)\in X\times X^*$. Suppose
that $x_0\in\dom A \cap \inte \dom B$ and that there exists a sequence $(a_n)_{n\in\NN}$ in $\dom A\cap\dom B$ such that $(a_n)_{n\in\NN}$ converges to a point in
$\left[x_0,z\right[$, and that
\begin{align}
a_n\in\bd\dom B,\quad\forall n\in\NN.\label{LeWExc:2E1}
\end{align}
Then
 $F_{A+B}(z,z^*)=+\infty$.
\end{lemma}

\begin{proof}
Suppose to the contrary that
\begin{align}
(z,z^*)\in\dom F_{A+B}.\label{LeWExc:2E2}
\end{align}
By the assumption, there exists $0\leq\lambda<1$ such that \begin{align}
a_n\longrightarrow x_0+\lambda(z-x_0).\label{LeWExc:2E3}
\end{align}
By the Separation Theorem and Fact~\ref{f:referee02c},
there exists $(y^*_n)_{n\in\NN}$ in  $X^*$ such that
$\|y^*_n\|=1$ and $y^*_n\in N_{\overline{\dom B}}(a_n)$.
Thus $ky^*_n\in N_{\overline{\dom B}}(a_n), \forall k>0$.
Since $x_0\in\inte\dom B$, there exists $\delta>0$ such that $x_0+\delta B_X\subseteq \dom B$. Thus
\begin{align*}
\big\langle y^*_n, a_n\big\rangle&\geq\sup\big\langle y^*_n, x_0+ \delta B_X\big\rangle
\geq \big\langle y^*_n, x_0\big\rangle+ \sup\big\langle y^*_n,\delta B_X\big\rangle=\big\langle y^*_n, x_0\big\rangle+\delta\|y^*_n\|\\
&=\big\langle y^*_n, x_0\big\rangle+\delta.
\end{align*}
Hence
\begin{align}
\big\langle y^*_n, a_n-x_0\big\rangle&\geq\delta.\label{FPCoS:es10a}
\end{align}
By the Banach-Alaoglu Theorem
(see \cite[Theorem~3.15]{Rudin}), there exists a weak$^*$ convergent and bounded subnet $(y^*_i)_{i\in O}$ such that
\begin{align}
y^*_i\weakstarly y^*_{\infty}\in X^*.\label{FPCoS:es11a}
\end{align}
Then \eqref{FPCoS:es10a} and \eqref{LeWExc:2E3} imply that
\begin{align*}
\big\langle y^*_{\infty}, \lambda (z-x_0)\big\rangle
\geq \delta.
\end{align*}
Thus, as before, $\lambda>0$ and
\begin{align}
\big\langle y^*_{\infty}, z-x_0\big\rangle
\geq \frac{\delta}{\lambda}>0.\label{FPCoS:es12a}
\end{align}
Since $B$ is maximally monotone, $B=B+N_{\overline{\dom B}}$.
As $a_n\in \dom A\cap\dom B$, we have
\begin{align*}
F_{A+B}(z,z^*)\geq\sup\left[
\big\langle z-a_n,  A(a_n)\big\rangle +
\big\langle z-a_n, B(a_n)+ ky^*_n\big\rangle+\langle z^*, a_n\big\rangle\right],\quad\forall n\in\NN,\forall k>0.
\end{align*}
Thus
\begin{align*}
\frac{F_{A+B}(z,z^*)}{k}\geq\sup\left[
\Big\langle z-a_n,  \frac{A(a_n)}{k}\Big\rangle +
\Big\langle z-a_n, \frac{B(a_n)}{k}+ y^*_n\Big\rangle+\frac{\langle z^*, a_n\big\rangle}{k}\right],\quad\forall n\in\NN,\forall k>0.
\end{align*}
Since $(z,z^*)\in\dom F_{A+B}$ by \eqref{LeWExc:2E2}, on letting $k\longrightarrow +\infty$ we obtain
\begin{align*}
0\geq
\big\langle z-a_n, y^*_n\big\rangle,\quad\forall n\in\NN.
\end{align*}
Combining with \eqref{FPCoS:es11a}, \eqref{LeWExc:2E3} and taking the limit along the bounded subnet in the above inequality, we have
\begin{align*}
0\geq
\big\langle (1-\lambda) (z-x_0), y^*_{\infty}\big\rangle.
\end{align*}
Since $\lambda<1$,
\begin{align*}
\big\langle z-x_0, y^*_{\infty}\big\rangle\leq0,
\end{align*}
which contradicts \eqref{FPCoS:es12a}.

Hence $F_{A+B}(z,z^*)=+\infty$.
\end{proof}

\begin{lemma}\label{LeWExc:3}
Let $A:X\To X^*$ be   of type (FPV). Suppose $x_0\in\dom A$ but
  that
$z\notin\overline{\dom A}$. Then there is a sequence
$(a_n, a^*_n)_{n\in\NN}$ in $\gra A$ so that $(a_n)_{n\in\NN}$ converges to a point in
$\left[x_0,z\right[$ and
\begin{align*}
\langle z-a_n, a^*_n\rangle\longrightarrow+\infty.
\end{align*}
\end{lemma}
\begin{proof}
Since $z\notin\overline{\dom A}$,
$z\neq x_0$. Thence there exist $\alpha>0$ and $y^*_0\in X^*$
such that $\langle y^*_0,z-x_0\rangle\geq\alpha$.
Set
 \begin{align*}
U_n:= [x_0, z]+\tfrac{1}{n}U_X,\quad \forall n\in\NN.
\end{align*}
Since $x_0\in\dom A$,
 $U_n\cap \dom A\neq\varnothing$.
Now $(z, ny^*_0)\notin\gra A$ and $z\in U_n$. As
$A$ is of type (FPV),
 there exist $(a_n,a^*_n)_{n\in\NN}$ in $\gra A$ with $a_n\in U_n$ such that
 \begin{align}
 \langle z-a_n,a^*_n\rangle>\langle z-a_n,ny^*_0\rangle.\label{LeWExc:3E1}
 \end{align}
 As $a_n\in U_n$, $(a_n)_{n\in\NN}$ has a subsequence convergent
 to an element in  $\left[x_0,z\right]$.  We can assume that
\begin{align}
a_n\longrightarrow x_0+\lambda(z-x_0),\quad \text{where}\quad 0\leq \lambda\leq 1,\label{LeWExc:3E2}
\end{align}
and since $z\notin\overline{\dom A}$, we have $\lambda<1$. Thus, $x_0+\lambda(z-x_0)\in\left
[x_0,z\right[$.

Thus by \eqref{LeWExc:3E2} and $
\langle z-x_0, y^*_0\rangle\geq\alpha >0$,
\begin{align*}
\big\langle z-a_n,y^*_0\big\rangle\longrightarrow(1-\lambda)\langle z-x_0 ,y^*_0\rangle\geq(1-\lambda)\alpha>0.
\end{align*}

Hence there exists $N_0\in\NN$ such that for every $n\geq N_0$
\begin{align}
\big\langle z-a_n,y^*_0\big\rangle\geq
\frac{ (1-\lambda)\alpha}{2}>0.
\end{align}
Appealing to \eqref{LeWExc:3E1}, we have
\begin{align*}
 \big\langle z-a_n,a^*_n\big\rangle
 >\frac{ (1-\lambda)\alpha}{2}n>0,\quad\forall n\geq N_0,
 \end{align*}
and so
$\big\langle z-a_n,a^*_n\big\rangle\longrightarrow+\infty$. This completes the proof.
\end{proof}

\section{Our main result}
\label{s:main}
Before we come to our main result, we need the following two technical results which let us place points in the  closures of the domains of $A$ and $B$.
The proof of Proposition~\ref{ProCVS:P1}
 follows  in part that of \cite[Theorem~3.4]{Yao3}.

\begin{proposition}\label{ProCVS:P1}
Let $A:X\To X^*$ be   of type (FPV), and let
$B:X\rightrightarrows X^*$ be maximally monotone. Suppose
that $\dom A \cap \inte \dom B\neq\varnothing$.
Let $(z,z^*)\in X\times X^*$ with $z\in\overline{\dom B}$.
Then
\begin{align*}F_{A+B}(z,z^*)\geq\langle z,z^*\rangle.
\end{align*}
\allowdisplaybreaks
\end{proposition}
\begin{proof}
Clearly, $F_{A+B}(z,z^*)\geq\langle z,z^*\rangle$
 if $(z,z^*)\notin\dom F_{A+B}$.
Now suppose that $(z,z^*)\in\dom F_{A+B}$.
We can suppose that $0\in \dom A \cap \inte \dom B$
 and $(0,0)\in\gra A\cap\gra B$.
Next, we show that
\begin{align}
F_{A+B}(t z,tz^*)\geq t^2\langle z,z^*\rangle
\quad\text{and}\quad tz\in\inte\dom B,
\quad\forall t\in\left]0,1\right[.\label{See:10}\end{align}
Fix $t\in\left]0,1\right[$.
As  $0\in\inte\dom B$, $z\in\overline{\dom B}$, Fact~\ref{f:referee02c}
 and \cite[Theorem~1.1.2(ii)]{Zalinescu} imply
\begin{align}
tz\in\inte\overline{\dom B}\label{ReAu:1},
\end{align}
and  Fact~\ref{f:referee02c} strengthens this to
\begin{align}tz\in\inte\dom B.
\end{align}

We again consider two cases.

\emph{Case 1}: $tz\in\dom A$.

On selecting $a^*\in A(tz), b^*\in B(tz)$,  the definition of the Fitzpatrick function \eqref{ff:def} shows
\begin{align*}
F_{A+B}(tz,tz^*)\geq\langle tz^*,tz\rangle
+\langle tz, a^*+b^*\rangle-\langle tz, a^*+b^*\rangle=
\langle tz,tz^*\rangle.
\end{align*}
Hence \eqref{See:10} holds.

\emph{Case 2}: $tz\notin\dom A$.

If $\langle z,z^*\rangle \leq0$, then $F_{A+B}(tz,tz^*)
\geq0\geq\langle tz,tz^*\rangle$ because
$(0,0)\in\gra A\cap\gra B$.  So we assume that
\begin{align}\langle z,z^*\rangle>0.\label{FPCoS:ea1}
\end{align}
We first show that
\begin{align}
tz\in\overline{\dom A}.\label{FPC:1}
\end{align}
Set
 \begin{align*}
U_n:= [0,t z]+\tfrac{1}{n}U_X,\quad \forall n\in\NN.
\end{align*}
Since $0\in\dom A$,
 $U_n\cap \dom A\neq\varnothing$.
Since $(tz, n z^*)\notin\gra A$ and $tz\in U_n$, while
$A$ is of type (FPV),
 there is $(a_n,a^*_n)_{n\in\NN}$ in $\gra A$ with $a_n\in U_n$ such that
 \begin{align}
 \langle t z,a^*_n\rangle>n\langle t z-a_n,z^*\rangle
 +\langle a_n,a^*_n\rangle.\label{FP:4}
 \end{align}
As $a_n\in U_n$, $(a_n)_{n\in\NN}$ has a subsequence convergent
 to an element in  $\left[0,t z\right]$.
  We can assume that
\begin{align}
a_n\longrightarrow \lambda z,\quad \text{where}\quad 0\leq \lambda\leq t.\label{FPC:2}
\end{align}
As $t z\in\inte\dom B$  also $\lambda z\in\inte\dom B$,and so appealing to Fact~\ref{extlem},
 there exist $N\in\NN$ and  $K>0$ such that
\begin{align}a_n\in\inte\dom B \quad\text{and}\quad \sup_{v^*\in B(a_n)}\| v^*\|
\leq K,\quad \forall n\geq N.\label{See:1a2}
\end{align}
We claim that \begin{align}\lambda=t.\label{FPCoS:e1}
\end{align}
Suppose to the contrary that  $0\leq\lambda<t$.  As
$(a_n,a^*_n)\in\gra A$ and \eqref{See:1a2} holds, for every $n\geq N$
\begin{align*}&F_{A+B}(z,z^*)\\
&\geq\sup_{\{ v^*\in B(a_n)\}}\left[\langle a_n,z^*\rangle
+\langle z,a_n^*\rangle-\langle a_n,a_n^*\rangle
+\langle  z-a_n, v^*\rangle \right]\\
&\geq\sup_{\{ v^*\in B(a_n)\}}\left[\langle a_n, z^*\rangle
+\langle z,a_n^*\rangle-\langle a_n,a_n^*\rangle
-K\|  z-a_n\| \right]\\
&\geq\langle a_n, z^*\rangle
+\langle z,a_n^*\rangle-\langle a_n,a_n^*\rangle
-K\|  z-a_n\|
\\
&>\langle a_n, z^*\rangle
+\tfrac{1}{t}n\langle t z-a_n,z^*\rangle
+\tfrac{1}{t}\langle a_n,a^*_n\rangle
-\langle a_n,a^*_n\rangle-K\|  z-a_n\|
 \quad\text{(by \eqref{FP:4})}\\
&\geq \langle a_n, z^*\rangle+\tfrac{1}{t}n\langle t z-a_n,z^*\rangle-K\|  z-a_n\|
\quad\text{(since $\langle a_n,a^*_n\rangle\geq0$
 by $(0,0)\in\gra A$ and $t\leq 1$)}.
 \end{align*}
 Divide by $n$ on both sides of the above
  inequality and take the limit with respect to  $n$. Since \eqref{FPC:2} and $F_{A+B}(z,z^*)<+\infty$, we obtain
 \begin{align*}
(1- \tfrac{\lambda}{t})\langle z,z^*\rangle=
\langle z- \tfrac{\lambda}{t}z,z^*\rangle \leq0.\end{align*}
Since $0\leq\lambda<t$, we obtain $\langle z,z^*\rangle\leq0$,
 which contradicts \eqref{FPCoS:ea1}.
Hence $\lambda=t$ and  by \eqref{FPC:2} $tz\in\overline{\dom A}$ so that
 \eqref{FPC:1} holds.

We next show that
\begin{align}
F_{A+B}(t z,tz^*)\geq t^2\langle z,z^*\rangle.\label{FPCoS:e2}\end{align}

Set
 \begin{align*}
H_n:= t z+\tfrac{1}{n}U_X,\quad \forall n\in\NN.
\end{align*}
Note that  $H_n\cap \dom A\neq\varnothing$, since  $tz\in\overline{\dom A}\backslash {\dom A}$ by \eqref{FPC:1}.

Because $(tz, t z^*)\notin\gra A$ and $tz\in H_n$, and
$A$ is of type (FPV),
 there exists $(b_n,b^*_n)_{n\in\NN}$ in $\gra A$
such that $b_n\in H_n$  and
\begin{align}
\langle t z,b^*_n\rangle+\langle b_n,t z^*\rangle-
\langle b_n,b^*_n\rangle>t^2\langle z,z^*\rangle,
\quad \forall n\in\NN.\label{see:20}
\end{align}
As $t z\in\inte\dom B$ and $b_n\longrightarrow t z$, by Fact~\ref{extlem},
 there exist $N_1\in\NN$ and  $M>0$ such that
\begin{align}b_n\in\inte\dom B \quad\text{and}\quad \sup_{v^*\in B(b_n)}\| v^*\|
\leq M,\quad \forall n\geq N_1.\label{FPCoS:e3}
\end{align}
We now compute
\begin{align}
F_{A+B}(t z,t z^*)
&\geq\sup_{\{ c^*\in B(b_n)\}}\left[\langle b_n,t z^*\rangle
+\langle t z,b_n^*\rangle-\langle b_n,b_n^*\rangle
+\langle t z-b_n, c^*\rangle \right],\quad \forall n\geq N_1\nonumber\\
&\geq\sup_{\{ c^*\in B(b_n)\}}\left[t^2\langle z,z^*\rangle
+\langle t z-b_n,c^*\rangle \right],\quad \forall n\geq N_1
\quad\text{(by \eqref{see:20})}\nonumber\\
&\geq\sup\left[t^2\langle z,z^*\rangle
-M\|t z-b_n\| \right],\quad \forall n\geq N_1
\quad\text{(by \eqref{FPCoS:e3})}.
\end{align}
Thus,
\[F_{A+B}(t z,t z^*) \geq t^2\langle z,z^*\rangle\]
because $b_n\longrightarrow t z$. Hence $
F_{A+B}(t z,t z^*)\geq t^2\langle z,z^*\rangle$.
Thus \eqref{FPCoS:e2} holds.

Combining  the above cases, we see that \eqref{See:10} holds.
Since $(0,0)\in\gra (A+ B)$ and $A+B$ is monotone,
 we have $F_{A+B}(0,0)=\langle 0,0\rangle=0$.
Since $F_{A+B}$ is convex, \eqref{See:10} implies that
\begin{align*}
tF_{A+B}(z,z^*)=t F_{A+B}(z,z^*)+(1-t)F_{A+B}(0,0)
\geq F_{A+B}(t z,tz^*)\geq t^2\langle z,z^*\rangle,
\quad\forall t\in\left]0,1\right[.
\end{align*}
Letting $t\longrightarrow 1^{-}$ in the above inequality, we obtain
$F_{A+B}(z,z^*)\geq \langle z,z^*\rangle$.
\end{proof}

We have one more block to put in place:

\begin{proposition}\label{ProCVS:P2}
Let $A:X\To X^*$ be   of type (FPV), and let
$B:X\rightrightarrows X^*$ be maximally monotone. Suppose
 $\sta(\dom A) \cap \inte \dom B\neq \varnothing$, and
 $(z,z^*)\in\dom F_{A+B}$.
Then
$z\in\overline{\dom A}$.

\end{proposition}
\begin{proof}
We can and do suppose that $0\in\sta(\dom A)\cap\inte\dom B$ and $(0,0)\in\gra A\cap\gra B$.
As before, we suppose to the contrary that
\begin{align}
z\notin\overline{\dom A}.\label{FPCoS:e5}
\end{align}
Then $z\neq0$. By the assumption that $z\notin\overline{\dom A}$,
 Lemma~\ref{LeWExc:3} implies that there
 exist $(a_n,a^*_n)_{n\in\NN}$ in $\gra A$ and $0\leq\lambda<1$ such that
 \begin{align}
 \langle z-a_n,a^*_n\rangle\longrightarrow+\infty\quad\text{and}\quad
 a_n\longrightarrow \lambda z.\label{FPCoS:e8}
 \end{align}

 We yet again consider two cases.

 \emph{Case 1}: There exists a subsequence of $(a_n)_{n\in\NN}$ in $\dom B$.

 We can suppose that $a_n\in \dom B$ for every $n\in\NN$.
 Thus by \eqref{FPCoS:e8} and Lemma~\ref{LeWExc:1},
we have  $F_{A+B}(z,z^*)=+\infty$, which contradicts our original assumption
that $(z,z^*)\in \dom F_{A+B}$.

 \emph{Case 2}: There exists $N_1\in\NN$ such that
 $a_n\not\in\dom B$ for every $n\geq N_1$.

  Now we can suppose that $a_n\notin \dom B$ for every $n\in\NN$.
Since $a_n\notin\dom B$,
  Fact~\ref{f:referee02c} and Fact~\ref{Ll:l1} shows that there  exists $\lambda_n\in\left[0,1\right]$ such that
\begin{align}
\lambda_n a_n\in\bd\overline{\dom B}.\label{FPCoS:es1}
\end{align}
By \eqref{FPCoS:e8}, we can  suppose that
\begin{align}
\lambda_n a_n\longrightarrow\lambda_{\infty} z.\label{FPCoS:es1d}
\end{align}
Since $0\in \sta(\dom A)$  and $a_n\in\dom A$, $\lambda_n a_n\in\dom A$.
Then \eqref{FPCoS:e8} implies that
 \begin{align}\lambda_{\infty}<1.\label{scvd:1}
 \end{align}

We further split Case 2 into two subcases.

\emph{Subcase 2.1}: There exists a subsequence of
 $(\lambda_n a_n)_{n\in\NN}$ in $\dom B$. We may again suppose $\lambda_n a_n\in\dom B$ for every $n\in\NN$.
Since $0\in\sta (\dom A)$ and $a_n
\in\dom A$, $\lambda_na_n\in\dom A$.  Then
by \eqref{FPCoS:es1} and \eqref{FPCoS:es1d}, \eqref{scvd:1} and Lemma~\ref{LeWExc:2},
$F_{A+B}(z,z^*)=+\infty$, which contradicts the hypothesis that $(z,z^*)\in\dom F_{A+B}$.

\emph{Subcase 2.2}: There exists $N_2\in\NN$ such that
 $\lambda_n a_n\not\in\dom B$ for every $n\geq N_2$.
We can now assume that $ \lambda_n a_n\not\in\dom B$ for every $n\in\NN$.
Thus $a_n\neq0$ for every $n\in\NN$.
Since $0\in\inte\dom B$, \eqref{FPCoS:es1} and \eqref{FPCoS:es1d}
imply that $0<\lambda_{\infty}$ and then by \eqref{scvd:1}
\begin{align}
0<\lambda_{\infty}<1.\label{PCSM:e1}
\end{align}
Since $0\in\inte\dom B$, \eqref{FPCoS:es1} implies that $\lambda_n>0$ for every $n\in\NN$.
By \eqref{FPCoS:e8}, $\|a_n-z\|\nrightarrow 0$.
Then we can and do suppose that $\|a_n-z\|\neq0$ for every $n\in\NN$.
Fix $n\in\NN$.
Since $0\in\inte\dom B$, there exists $0<\rho_0\leq1$ such that
$\rho_0 B_X\subseteq \dom B$. As $0\in\sta (\dom A)$
and $a_n\in\dom A$, $\lambda_na_n\in\dom A$.
Set \begin{align}
b_n:=\lambda_na_n\quad\text{ and take} \quad b^*_n\in A(\lambda_na_n).
\label{PCSM:e2}
\end{align}
Next we show that there exists $\varepsilon_n\in\left]0,\frac{1}{n}\right[$ such that with
$H_n:=(1-\varepsilon_n)b_n+\varepsilon_n \rho_0U_X$ and $\tau_0:=\frac{1}{\lambda_n}\left[2\|z\|+2\|a_n\|+2
+(\|a_n\|+1)\frac{2\lambda_n\|z-a_n\|}{\rho_0}\right]$, we have
\begin{align}
H_n\subseteq\dom B~\mbox{and~}
\inf\big\|B\big(H_n\big)\big\|\geq n(1+\tau_0\|b_n^*\|),~\mbox{while~}\varepsilon_n\max\{\|a_n\|,1\}<\frac{1}{2}\|z-a_n\|\lambda_n.
\label{PCSM:c1}\end{align}

For every $s\in \left]0,1\right[$,  \eqref{FPCoS:es1} and Fact~\ref{f:referee02c} imply that
$(1-s)b_n+s \rho_0 B_X\subseteq \overline{\dom B}$.
By Fact~\ref{f:referee02c} again,
$(1-s)b_n+s \rho_0 U_X\subseteq \inte\overline{\dom B}=\inte\dom B$.

Now we show the second assertion of \eqref{PCSM:c1}.  Let $k\in\NN$ and $(s_k)_{k\in\NN}$ be  a positive sequence such that  $s_k\longrightarrow 0$
when $k\longrightarrow \infty$. It suffices to show
\begin{align}
\lim_{k\rightarrow\infty}\inf\big\|B\big((1-s_k)b_n+s_k \rho_0U_X\big)\big\|=+\infty.
\label{PCSM:e7}
\end{align}
Suppose to the contrary there exist a sequence $(c_k,c^*_k)_{k\in\NN}$ in $\gra B\cap \left[\big((1-s_k)b_n+s_k \rho_0U_X\big)\times X^*\right]$ and $L>0$ such that $\sup_{k\in\NN}\|c^*_k\|\leq L$. Then
$c_k\longrightarrow b_n=\lambda_n a_n$.
By the Banach-Alaoglu Theorem
(again see \cite[Theorem~3.15]{Rudin}),
 there  exist a weak* convergent subnet, $(c^*_{\beta})_{\beta\in J}$ of
$(c^*_k)_{k\in\NN}$ such that $c^*_{\beta}\weakstarly c^*_{\infty}\in X^*$.
\cite[Corollary~4.1]{BY1} shows that $(\lambda_n a_n, c^*_{\infty})\in\gra B$,
which contradicts our assumption that $\lambda_na_n\notin \dom B$.
Hence \eqref{PCSM:e7} holds and so does \eqref{PCSM:c1}.

Set $t_n:=\frac{\varepsilon_n\rho_0}{2\lambda_n\|z-a_n\|}$ and thus $0<t_n<\frac{1}{4}$. Thus
\begin{align}
t_n \lambda_{n} z+(1-t_n)(1-\varepsilon_n)b_n\in H_n.\label{PCM:c1}
\end{align}
Next we show there exists $(\widetilde{a_n}, \widetilde{a_n}^*)_{n\in\NN}$
in $\gra A\cap (H_n\times X^*)$ such that
\begin{align}
\big\langle z-\widetilde{a_n},\widetilde{a_n}^*\big\rangle\geq-\tau_0\|b_n^*\|.\label{PCSM:c2}
\end{align}
We consider two further subcases.

\emph{Subcase 2.2a}:
$\big(t_n \lambda_{n} z+(1-t_n)(1-\varepsilon_n)b_n, (1+t_n)b^*_n\big)\in\gra A$.
Set $(\widetilde{a_n}, \widetilde{a_n}^*):=\big(t_n \lambda_{n} z+(1-t_n)(1-\varepsilon_n)b_n, (1+t_n)b^*_n\big)$. Since $(0,0)\in\gra A$,
$\langle b_n, b^*_n\rangle\geq0$.

 Then we have
\begin{align}
&\Big\langle t_n \lambda_{n} z-\widetilde{a_n},\, \widetilde{a_n}^*\Big\rangle
=\Big\langle t_n \lambda_{n} z-t_n \lambda_{n} z-(1-t_n)(1-\varepsilon_n)b_n,\, (1+t_n)b^*_n\Big\rangle\nonumber\\
&=\Big\langle-(1-t_n)(1-\varepsilon_n)b_n, (1+t_n)b^*_n\Big\rangle=
-\Big\langle(1-t^2_n)(1-\varepsilon_n)b_n, b^*_n\Big\rangle\geq
-\Big\langle b_n, b^*_n\Big\rangle.\label{PCSM:e3}
\end{align}
On the other hand, \eqref{PCSM:e2} and the monotonicity of $A$ imply that
\begin{align*}
\Big\langle t_n \lambda_{n} z+(1-t_n)(1-\varepsilon_n)b_n-b_n, t_nb^*_n\Big\rangle
=\Big\langle t_n \lambda_{n} z+(1-t_n)(1-\varepsilon_n)b_n-b_n, (1+t_n)b^*_n-b^*_n\Big\rangle\geq0.
\end{align*}
Thus
\begin{align}
\Big\langle t_n \lambda_{n} z-\left[1-(1-t_n)(1-\varepsilon_n)\right]b_n, b^*_n\Big\rangle
\geq0.\label{PCSM:e4}
\end{align}
Since $1-(1-t_n)(1-\varepsilon_n)>0$ and $\langle b_n, b_n^*\rangle
=\langle b_n-0, b_n^*-0\rangle\geq0$,
\eqref{PCSM:e4} implies that
$\langle t_n \lambda_{n} z,b^*_n\rangle\geq0$ and thus
\begin{align*}\langle z,b^*_n\rangle\geq0.
\end{align*}
Then by $\widetilde{a_n}^*=(1+t_n)b^*_n$ and $t_n\lambda_n\leq 1$,
\eqref{PCSM:e3} implies that
\begin{align*}
\Big\langle z-\widetilde{a_n},\, \widetilde{a_n}^*\Big\rangle\geq
-\Big\langle b_n,\, b^*_n\Big\rangle\geq-\|b_n\|\cdot\|b^*_n\|\geq-\|a_n\|\cdot\|b^*_n\|\geq-\tau_0\|b_n^*\|.
\end{align*}
Hence \eqref{PCSM:c2} holds.

\emph{Subcase 2.2b}: $\big(t_n \lambda_{n} z+(1-t_n)(1-\varepsilon_n)b_n, (1+t_n)b^*_n\big)\notin\gra A$.
By $0\in\sta (\dom A)$ and $a_n\in\dom A$, we have $
(1-\varepsilon_n)\lambda_na_n\in\dom A$, hence $\dom A\cap H_n\neq\varnothing$.
Since $
t_n \lambda_{n} z+(1-t_n)(1-\varepsilon_n)b_n\in H_n$ by \eqref{PCM:c1}, $\big(t_n \lambda_{n} z+(1-t_n)(1-\varepsilon_n)b_n, (1+t_n)b^*_n\big)\notin\gra A$ and  $A$ is of type (FPV),
there exists $(\widetilde{a_n}, \widetilde{a_n}^*)\in\gra A$ such that
$\widetilde{a_n}\in H_n$ and
\allowdisplaybreaks
\begin{align*}
&\Big\langle t_n \lambda_{n} z+(1-t_n)(1-\varepsilon_n)b_n
-\widetilde{a_n},
\,\widetilde{a_n}^*-(1+t_n)b^*_n\Big\rangle>0\\
&\Rightarrow\Big\langle t_n \lambda_{n} z-\left[1-(1-t_n)(1-\varepsilon_n)\right]\widetilde{a_n}+(1-t_n)(1-\varepsilon_n)(b_n
-\widetilde{a_n}),
\,\widetilde{a_n}^*-b^*_n\Big\rangle\\
&\quad>\Big\langle t_n \lambda_{n} z+(1-t_n)(1-\varepsilon_n)b_n
-\widetilde{a_n},
\,t_nb^*_n\Big\rangle\geq \Big\langle t_n \lambda_{n} z
-\widetilde{a_n},
\,t_nb^*_n\Big\rangle\quad\text{(since $\langle b_n, b^*_n\rangle\geq0$)}\\
&\Rightarrow\Big\langle t_n \lambda_{n} z-\left[1-(1-t_n)(1-\varepsilon_n)\right]\widetilde{a_n},
\,\widetilde{a_n}^*-b^*_n\Big\rangle\\
&\quad>\Big\langle (1-t_n)(1-\varepsilon_n)(b_n
-\widetilde{a_n}),
\,b^*_n-\widetilde{a_n}^*\Big\rangle+\Big\langle t_n \lambda_{n} z
-\widetilde{a_n},
\,t_nb^*_n\Big\rangle\nonumber\\
&\Rightarrow\Big\langle t_n \lambda_{n} z-\left[1-(1-t_n)(1-\varepsilon_n)\right]\widetilde{a_n},
\,\widetilde{a_n}^*-b^*_n\Big\rangle>\Big\langle t_n \lambda_{n} z
-\widetilde{a_n},
\,t_nb^*_n\Big\rangle\nonumber\\
&\Rightarrow \Big\langle t_n \lambda_{n} z-\left[t_n+\varepsilon_n-t_n\varepsilon_n\right]\widetilde{a_n},
\,\widetilde{a_n}^*\Big\rangle>\Big\langle t_n \lambda_{n} z
-\widetilde{a_n},
\,t_nb^*_n\Big\rangle+\Big\langle t_n \lambda_{n} z-\left[t_n+\varepsilon_n-t_n\varepsilon_n\right]\widetilde{a_n},
\,b^*_n\Big\rangle.
\end{align*}
Since $\langle \widetilde{a_n}, \widetilde{a_n}^*\rangle=
\langle \widetilde{a_n}-0, \widetilde{a_n}^*-0\rangle\geq0$ and $
t_n+\varepsilon_n-t_n\varepsilon_n\geq t_n\geq t_n\lambda_n$,
$\Big\langle \left[t_n+\varepsilon_n-t_n\varepsilon_n\right]\widetilde{a_n},
\,\widetilde{a_n}^*\Big\rangle\geq t_n\lambda_n\Big\langle\widetilde{a_n},
\,\widetilde{a_n}^*\Big\rangle$.  Thus
\begin{align*}
&\Big\langle t_n \lambda_{n} z-t_n\lambda_n\widetilde{a_n},
\,\widetilde{a_n}^*\Big\rangle>\Big\langle t_n \lambda_{n} z
-\widetilde{a_n},
\,t_nb^*_n\Big\rangle+\Big\langle t_n \lambda_{n} z-\left[t_n+\varepsilon_n-t_n\varepsilon_n\right]\widetilde{a_n},
\,b^*_n\Big\rangle\nonumber\\
&\Rightarrow \Big\langle \frac{t_n \lambda_{n} z-t_n\lambda_n\widetilde{a_n}}{\lambda_nt_n},
\,\widetilde{a_n}^*\Big\rangle>\Big\langle t_n \lambda_{n} z
-\widetilde{a_n},
\,\frac{1}{\lambda_n}b^*_n\Big\rangle+\Big\langle \frac{t_n \lambda_{n} z-\left[t_n+\varepsilon_n-t_n\varepsilon_n\right]\widetilde{a_n}}{\lambda_nt_n},
\,b^*_n\Big\rangle\\
&\Rightarrow \Big\langle   z-\widetilde{a_n},
\,\widetilde{a_n}^*\Big\rangle>\Big\langle t_n \lambda_{n} z
-\widetilde{a_n},
\,\frac{1}{\lambda_n}b^*_n\Big\rangle+\Big\langle z-\left[1+\frac{\varepsilon_n}{t_n}
-\varepsilon_n\right]\frac{1}{\lambda_n}\widetilde{a_n},
\,b^*_n\Big\rangle\\
&\Rightarrow \Big\langle  z-\widetilde{a_n},
\,\widetilde{a_n}^*\Big\rangle>
-\frac{1}{\lambda_n}\|b^*_n\|\big(\|z\|+\|a_n\|+1\big)
-\|b^*_n\|\Big(\| z\|+\frac{1}{\lambda_n}(\|a_n\|+1)\big(1+\frac{2\lambda_n\|z-a_n\|}{\rho_0}
\Big)\nonumber\\
&\Rightarrow \Big\langle  z-\widetilde{a_n},
\,\widetilde{a_n}^*\Big\rangle>
-\|b^*_n\|\frac{1}{\lambda_n}\left[2\|z\|+2\|a_n\|+2
+(\|a_n\|+1)\frac{2\lambda_n\|z-a_n\|}{\rho_0}
\right]=-\tau_0\|b^*_n\|.
\end{align*}
Finally, combining all the subcases, we deduce that \eqref{PCSM:c2} holds.

Since $\varepsilon_n<\frac{1}{n}$ and $\widetilde{a_n}\in H_n$,
\eqref{FPCoS:es1d} shows that
\begin{align}
\widetilde{a_n}\longrightarrow \lambda_{\infty}z.\label{PCSM:e6}
\end{align}

Take $w^*_n\in B(\widetilde{a_n})$ by \eqref{PCSM:c1}. Then by \eqref{PCSM:c1} again,
\begin{align}
\|w^*_n\|\geq n(1+\tau_0\|b^*_n\|),\quad\forall n\in\NN.\label{PCSM:e5}
\end{align}
Then by \eqref{PCSM:c2}, we have
\begin{align*}
-\tau_0\|b^*_n\|+\Big\langle z-\widetilde{a_n}, w^*_n\Big\rangle+\Big\langle z^*, \widetilde{a_n}\Big\rangle&\leq
\Big\langle z-\widetilde{a_n}, \widetilde{a_n}^*\Big\rangle+
\Big\langle z-\widetilde{a_n}, w^*_n\Big\rangle+\Big\langle z^*, \widetilde{a_n}\Big\rangle\\
&\leq F_{A+B}(z,z^*)
\end{align*}
Thus
\begin{align}
-\frac{\tau_0\|b^*_n\|}{\|w^*_n\|}+\Big\langle z-\widetilde{a_n}, \frac{w^*_n}{\|w^*_n\|}\Big\rangle+\Big\langle \frac{z^*}{\|w^*_n\|}, \widetilde{a_n}\Big\rangle
&\leq \frac{F_{A+B}(z,z^*)}{\|w^*_n\|}.\label{PCSM:c3}
\end{align}
By the Banach-Alaoglu Theorem
(see \cite[Theorem~3.15]{Rudin}),
 there  exist a weak* convergent subnet, $(\frac{w^*_i}{\|w^*_i\|})_{i\in I}$ of
$\frac{w^*_n}{\|w^*_n\|}$ such that
\begin{align}
\frac{w^*_i}{\|w^*_i\|}\weakstarly w^*_{\infty}\in X^*.\label{PCSM:c4}
\end{align}
Combine \eqref{PCSM:e6}, \eqref{PCSM:e5} and \eqref{PCSM:c4}, by $F_{A+B}(z,z^*)<+\infty$, and take the limit along the subnet in \eqref{PCSM:c3} to obtain
\begin{align*}
\Big\langle z-\lambda_{\infty} z, w^*_{\infty}\Big\rangle\leq0.
\end{align*}
Then  \eqref{PCSM:e1} shows that
\begin{align}
\Big\langle z, w^*_{\infty}\Big\rangle\leq0.\label{PCSM:c6}
\end{align}
On the other hand, since $0\in\inte\dom B$, Fact~\ref{extlem} implies that there exists
$\rho_1>0$ and $M>0$ such that
\begin{align*}
\Big\langle \widetilde{a_n}, w^*_n\Big\rangle\geq\rho_1\|w^*_n\|-(\|\widetilde{a_n}\|
+\rho_1)M.
\end{align*}
Thus
\begin{align*}
\Big\langle \widetilde{a_n}, \frac{w^*_n}{\|w^*_n\|}\Big\rangle\geq\rho_1-\frac{(\|\widetilde{a_n}\|
+\rho_1)M}{\|w^*_n\|}.
\end{align*}
Use \eqref{PCSM:e6}, \eqref{PCSM:e5} and \eqref{PCSM:c4}, and
 take the limit along the subnet in the above inequality to obtain
\begin{align*}
\Big\langle \lambda_{\infty}z, w^*_{\infty}\Big\rangle\geq\rho_1.
\end{align*}
Hence
\begin{align*}
\Big\langle z, w^*_{\infty}\Big\rangle\geq\frac{\rho_1}{\lambda_{\infty}}>0,
\end{align*}
which contradicts \eqref{PCSM:c6}.

  Combining all the above cases, we have arrived at $z\in\overline{\dom A}$.
\end{proof}

We are finally ready to prove our main result. The special case in which  $B$ is the normal cone operator of a nonempty closed convex set
 was first established by Voisei in \cite{Voisei09}.

\begin{theorem}[(FPV) Sum Theorem]\label{TePGV:1}
Let $A, B:X\To X^*$ be  maximally monotone
with $\sta(\dom A)\cap\inte\dom B\neq\varnothing$. Assume
that $A$ is of type (FPV). Then
$A+B$ is maximally monotone.
\end{theorem}
\begin{proof}
After translating the graphs if necessary, we can and do assume that
$0\in\sta(\dom A)\cap\inte\dom B$ and that $(0,0)\in\gra A\cap\gra B$.
By Fact~\ref{f:Fitz}, $\dom A\subseteq P_X(\dom F_A)$ and
 $\dom B\subseteq P_X(\dom F_{B})$.
Hence,
\begin{align}\bigcup_{\lambda>0} \lambda
\big(P_X(\dom F_A)-P_X(\dom F_{B})\big)=X.\end{align}
Thus, by Fact~\ref{f:referee1}, it suffices to show that
\begin{equation} \label{e0:ourgoal}
F_{A+ B}(z,z^*)\geq \langle z,z^*\rangle,\quad \forall(z,z^*)\in X\times X^*.
\end{equation}
Take $(z,z^*)\in X\times X^*$.
Then
\begin{align}
&F_{A+B}(z,z^*)\nonumber\\
&=\sup_{\{x,x^*,y^*\}}\left[\langle x,z^*\rangle+\langle z-x,x^*\rangle
+\langle z-x, y^*\rangle -\iota_{\gra A}(x,x^*)-\iota_{\gra
B}(x,y^*)\right].\label{see:1}
\end{align}
Suppose to the contrary that there exists $\eta>0$ such that
\begin{align}
F_{A+B}(z,z^*)+\eta<\langle z,z^*\rangle,\label{FPCoS:e20}
\end{align}
so that
\begin{align}
(z,z^*)\,\text{ is monotonically related to $\gra (A+B)$}.\label{SDFC:47}\end{align}

Then  by Proposition~\ref{ProCVS:P1} and Proposition~\ref{ProCVS:P2} ,
\begin{align}z\in\overline{\dom A}\backslash \overline{\dom B}.\label{TCoSM:e1}
\end{align}
Now  by Lemma~\ref{rcf:01},
\begin{align}  z\notin\dom A.\label{TCoSM:d1}
\end{align}
Indeed, if $z\in\dom A$,  Lemma~\ref{rcf:01} and \eqref{SDFC:47} show that
$z\in\dom B$. Thus, $z\in\dom A\cap\dom B$ and hence
$F_{A+B}(z,z^*)\geq\langle z,z^*\rangle$ which contradicts
\eqref{FPCoS:e20}. Thence we have established \eqref{TCoSM:d1}.

Thus \eqref{TCoSM:e1} implies that
there exists $(a_n, a^*_n)_{n\in\NN}$ in $\gra A$ such that
\begin{align}
a_n\longrightarrow z.\label{TCoSM:e2}
\end{align}
By \eqref{TCoSM:e1}, $a_n\notin\overline{\dom B}$  for all but finitely many terms $a_n$. We can suppose that
$a_n\notin\overline{\dom B}$ for all $n\in\NN$.
Fact~\ref{f:referee02c} and Fact~\ref{Ll:l1} show that there  exists $
\lambda_n\in\left]0,1\right[$ such that
\begin{align}
\lambda_n a_n\in\bd\overline{\dom B}.\label{TCoSM:e3}
\end{align}
By \eqref{TCoSM:e2}, we can assume that
\begin{align}\lambda_n\longrightarrow \lambda_{\infty}\in\left[0,1\right]\quad\text{and thus}\quad
\lambda_n a_n\longrightarrow\lambda_{\infty} z.\label{TCoSM:e4}
\end{align}
Then  by \eqref{TCoSM:e3} and \eqref{TCoSM:e1}\begin{align}
\lambda_{\infty}<1.\label{TCoSM:da1}
\end{align}

We consider two cases.

\emph{Case 1}: There exists a subsequence of
 $(\lambda_n a_n)_{n\in\NN}$ in $\dom B$.

We can  suppose that $\lambda_n a_n\in\dom B$ for every $n\in\NN$.
Since $0\in\sta(\dom A)$  and $a_n
\in\dom A$, $\lambda_na_n\in\dom A$.  Then
by \eqref{TCoSM:e3},\eqref{TCoSM:e4}, \eqref{TCoSM:da1} and Lemma~\ref{LeWExc:2},
$F_{A+B}(z,z^*)=+\infty$,
which contradicts \eqref{FPCoS:e20} that $(z,z^*)\in\dom F_{A+B}$.

\emph{Case 2}: There exists $N\in\NN$ such that
 $\lambda_n a_n\not\in\dom B$ for every $n\geq N$.

We can suppose that $ \lambda_n a_n\not\in\dom B$ for every $n\in\NN$.
Thus $a_n\neq0$ for every $n\in\NN$.
Following the pattern of Subcase 2.2 in the proof of Proposition~\ref{ProCVS:P2}
\footnote{We  banish the details  to Appendix~\ref{Appenc:1} to spare the readers.}, we obtain a contradiction.

Combing all the above cases, we have $F_{A+B}(z,z^*)\geq\langle z, z^*\rangle$ for all $(z,z^*)\in X\times X^*$.
Hence $A+B$ is  maximally monotone.
\end{proof}

\begin{remark}
In Case 2 in the proof of Theorem~\ref{TePGV:1} (see Appendix~\ref{Appenc:1} below), we use Lemma~\ref{rcf:01} to deduce that
$\|a_n-z\|\neq0$.  Without the help of Lemma~~\ref{rcf:01}, we may still can obtain \eqref{TCoSM:d8} as follows.
For the case of $a_n=z$, consider whether $\big((1-\varepsilon_n)b_n,0\big)=\big((1-\varepsilon_n)\lambda_nz,0\big)\in H_n\times X^*$ is in $\gra A$ or not.
We can deduce that there exists $(\widetilde{a_n}, \widetilde{a_n}^*)_{n\in\NN}$
in $\gra A\cap (H_n\times X^*)$ such that
\begin{align*}
\big\langle z-\widetilde{a_n},\widetilde{a_n}^*\big\rangle\geq0.
\end{align*}
Hence \eqref{TCoSM:d8} holds,
and the proof of Theorem~\ref{TePGV:1} can be achieved without Lemma~~\ref{rcf:01}.
\qede
\end{remark}

\section{Examples and Consequences}\label{s:cor}

We start by illustrating that the starshaped hypothesis catches operators whose domain may be non-convex and have no algebraic interior.

\begin{example}[Operators with starshaped domains]
We illustrate that there are many choices of
maximally monotone operator $A$ of type (FPV) with non-convex domain
such that $^{ic}\dom A =\inte\dom A=\varnothing $ and $\sta(\dom A)\neq\varnothing$.
Let $f:\RR^2\rightarrow\RX$ be defined by
\begin{align*}
(x,y)\mapsto\begin{cases}\max\{1-\sqrt{x},\,|y|\}\,&\text{if}\,x\geq0;\\
+\infty,&\text{otherwise}.
\end{cases}
\end{align*}
Consider an infinite dimensional Banach space $X$ containing  a nonempty closed and convex set $C$ such that $^{ic}C=\varnothing$. It is not known whether all spaces have this property but all separable or reflexive spaces certainly do \cite{BorVan}.
Define $A:(\RR^2\times X)\rightrightarrows(\RR^2\times X^*)$ by
\begin{align*}
(v,w)\rightrightarrows\Big(\partial f(v),
\partial \iota_{C}(w)\Big)=\partial F(v,w),
\end{align*}
where $F:=f\oplus\iota_{C}$.
Define $\|\cdot\|$ on $\RR^2\times X$ by $\|(v,w)\|:=\|v\|+\|w\|$.

Then $f$ is proper convex and lowers semicontinuous and so, therefore, is $F$. Indeed,
\cite[Example before Theorem~23.5, page~218]{Rock70CA}
shows that $\dom\partial f$ is not convex and consequently $\dom A$ is not convex. (Many other candidates for $f$ are given in \cite[Chapter 7]{BorVan}.)
Clearly, $A=\partial F$ is maximally monotone. Let $w_0\in C$ and $v_0=(2,0)$.
 Consider $(v_0,w_0)\in\RR^2\times X$.
 Since $v_0=(2,0)\in\inte\dom \partial f$, $v_0\in\sta(\dom\partial f)$ since $\dom f$ is convex.
 Thus $(v_0,w_0)\in\sta(\dom A)$. Since $^{ic}C=\varnothing$ and so $\inte C=\varnothing$, it follows that
 $^{ic}\dom A=\inte\dom A=\varnothing$. \cite[Theorem~48.4(d)]{Si2} shows that $A=\partial F$ is of type (FPV).\qede
\end{example}

The next example gives all the details of how to associate the support points of a convex set to a subgradient. In \cite{Holmes},   \cite{BroRoc} and \cite[Exercise~8.4.1, page~401]{BorVan} the construction is used to build empty subgradients in various Fr\'echet spaces and incomplete normed spaces.

\begin{example}[Support points]
Suppose that $X$ is separable.
We can always find  a  compact convex set $C\subseteq X$  such that $\spand C\neq X$  and $\overline{\spand C}=X$ \cite{BorVan}.
 Take $x_0 \notin \spand C$. Define $f:X\rightarrow\RX$ by
\begin{align}
f(x) :=\min \{t\in\RR\mid x+t x_0 \in C\},\quad\forall x\in X.
\end{align}
By direct computation
$f$ is proper lower semicontinuous and convex, see \cite{Holmes}.
By the definition of $f$, $\dom f=C+\RR x_0$.
Let $t\in\RR$ and $c\in C$. We shall establish that
\begin{align}
\partial f(t x_0+c)=\begin{cases}N_C(c) \cap \{y^*\in X^*\mid \langle y^*,x_0\rangle=-1\},
\,&\text{if\, $c\in\supp C$};\\
\varnothing, &\text{otherwise}.
\end{cases}.\label{Examhom:1}
\end{align}
Thence, also $\dom \partial f= \RR x_0+\supp C$.

First we show that the implication\begin{align}
tx_0+c=sx_0+d, \quad \text{where}\,t,s\in\RR, c,d\in C\quad\Rightarrow\quad
t=s \quad\text{and}\quad c=d\label{Examhom:a1}
\end{align}
holds.
Let $t,s\in\RR$ and  $c,d\in C$.
We have
$(t-s)x_0=d-c\in\spand C-\spand C=\spand C$. Since $x_0\notin \spand C$, $t=s$ and then
$c=d$. Hence we obtain \eqref{Examhom:a1}.

By \eqref{Examhom:a1}, we have
\begin{align}
f(tx_0+c)=-t,\quad\forall t\in\RR, \forall c\in C.\label{Examhom:b1}
\end{align}
 We  next show that
\eqref{Examhom:1} holds.

Since $\dom f=C+\RR x_0$,  by \eqref{Examhom:b1}, we have
\begin{align*}
&x^*\in \partial f(t x_0+c)\\
&\Leftrightarrow
\Big\langle x^*, sx_0+d-(t x_0+c)\Big\rangle\leq f(sx_0+d)-f(tx_0+c)=-s+t,\quad\forall s\in\RR, \forall d\in C\\
&\Leftrightarrow \Big\langle x^*, (s-t)x_0+(d-c)\Big\rangle\leq -s+t,\quad\forall s\in\RR, \forall d\in C\\
&\Leftrightarrow \Big\langle x^*, sx_0+(d-c)\Big\rangle\leq -s,\quad\forall s\in\RR, \forall d\in C\\
&\Leftrightarrow \Big\langle x^*, sx_0\Big\rangle\leq -s\quad\text{and}\quad\Big\langle x^*, d-c\Big\rangle\leq 0,\quad\forall s\in\RR, \forall d\in C\\
&\Leftrightarrow \Big\langle x^*, sx_0\Big\rangle\leq -s\quad\text{and}\quad
x^*\in N_C(c), x^*\neq0,\quad\forall s\in\RR\\
&\Leftrightarrow \Big\langle x^*, x_0\Big\rangle=-1,\quad
x^*\in N_C(c)\quad\text{and}\quad c\in\supp C.
\end{align*}
Hence \eqref{Examhom:1} holds.

As a concrete example of $C$ consider, for $1\le p < \infty$,  any order interval $C:=\{x \in \ell^p(\mathbb{N}) \colon \alpha \le x \le \beta\}$ where $\alpha <\beta \in \ell^p(\mathbb{N})$.
The example extends to all weakly compactly generated (WCG) spaces \cite{BorVan} with a weakly compact convex set in the role of $C$.
\qede
\end{example}

We gave the last example in part as it allows one to  better understand what the domain of a maximally monotone operator with empty interior can look like. While the star may be empty, it has been recently proven  \cite{Ves}, see also \cite{DeBVes}, that  for a closed convex function $f$ the domain of $\partial f$ is always pathwise and locally pathwise connected.

An immediate corollary of Theorem \ref{TePGV:1}  is the following which generalizes \cite[Corollary~3.9]{Yao3}.

\begin{corollary}[Convex domain]\label{CPGV:1}
Let $A, B:X\To X^*$ be  maximally monotone
with $\dom A\cap\inte\dom B\neq\varnothing$. Assume
that $A$ is of type (FPV) with convex domain. Then
$A+B$ is maximally monotone.
\end{corollary}

 An only slightly less immediate corollary is given next.

\begin{corollary}[Nonempty interior] \emph{(See \cite[Theorem~9(i)]{Bor2} or Fact~\ref{voiZalsm}.)}
\label{sumprci}
Let $A, B:X\To X^*$ be  maximally monotone
with $\inte\dom A\cap\inte\dom B\neq\varnothing$.
Then
$A+B$ is maximally monotone.
\end{corollary}
\begin{proof}
By the assumption, there exists $x_0\in\inte\dom A\cap\inte\dom B$.
We first show that $A$ is of type (FPV).
Let $C$ be a nonempty closed convex subset of $X$,
and suppose that $\dom A\cap \inte C \neq \varnothing$.
Let $x_1\in \dom A\cap\inte C$.
Fact~\ref{f:referee02c} and \cite[Theorem~1.1.2(ii)]{Zalinescu} imply
that $\left[x_0,x_1\right[\subseteq\inte\overline{\dom A}=\inte\dom A$.
Since $x_1\in\inte C$, there exists $0<\delta<1$ such that $x_1+\delta(x_0-x_1)\in
\inte\dom A\cap\inte C$.  Then $N_C+A$ is maximally monotone by Corollary~\ref{CPGV:1}
and \cite[Theorem~48.4(d)]{Si2}.
Hence by Fact~\ref{f:referee02a}, $A$ is of type (FPV), see also \cite{Bor2}.

Since $x_0\in\inte\dom A$, Fact~\ref{f:referee02c} and \cite[Theorem~1.1.2(ii)]{Zalinescu} imply that $x_0\in\sta (\dom A)$
and hence we have $x_0\in\sta(\dom A)\cap\inte\dom B$. Then by Theorem~\ref{TePGV:1},
we deduce that $A+B$ is maximally monotone.
\end{proof}

\begin{corollary}[Linear relation]\emph{(See \cite[Theorem~3.1]{BY3}.)}
Let $A:X\To X^*$ be a maximally monotone linear relation, and let
$B: X\rightrightarrows X^*$ be maximally monotone. Suppose  that
$\dom A\cap\inte\dom B\neq\varnothing$.  Then $A+B$ is maximally
monotone.
\end{corollary}

\begin{proof}
Apply Fact~\ref{f:referee01} and Corollary~\ref{CPGV:1} directly.
\end{proof}

The proof of our final Corollary~\ref{domain:L1} is adapted from that of \cite[Corollary~2.10]{Yao3} and \cite[Corollary~3.3]{BY3}. Moreover, it generalizes  both \cite[Corollary~2.10]{Yao3} and \cite[Corollary~3.3]{BY3}.

\begin{corollary}[FPV property of the sum]\label{domain:L1}
Let $A, B:X\To X^*$ be  maximally monotone
with $\dom A\cap\inte\dom B\neq\varnothing$. Assume
that $A$ is of type (FPV) with convex domain.
Then $A+B$ is of type  $(FPV)$.

\end{corollary}
\begin{proof} By Corollary~\ref{CPGV:1}, $A+B$ is maximally monotone.
Let $C$ be a nonempty closed convex subset of $X$,
and suppose that $\dom (A+B) \cap \inte C\neq \varnothing$.
Let $x_1\in \dom A \cap \inte\dom B$ and $x_2\in \dom (A+B) \cap \inte C$.
Then $x_1,x_2\in\dom A$, $x_1\in\inte\dom B$ and
$x_2\in\dom B\cap \inte C$.
Hence $\lambda x_1+(1-\lambda)x_2\in\inte\dom B$ for every $\lambda\in
\left]0,1\right]$ by Fact~\ref{f:referee02c} and \cite[Theorem~1.1.2(ii)]{Zalinescu} and so there exists $\delta\in\left]0,1\right]$ such that
$\lambda x_1+(1-\lambda)x_2\in\inte C$ for every $\lambda\in\left[0,\delta\right]$.

Thus, $\delta x_1+(1-\delta)x_2\in\dom A\cap\inte\dom B\cap\inte C$.
By Corollary~\ref{sumprci}, $B+N_C$ is maximally monotone.
Then, by Corollary~\ref{CPGV:1} (applied  $A$ and $B+N_C$ to $A$ and $B$),
$A+B+N_C=A+(B+N_C)$ is maximally monotone.
By Fact~\ref{f:referee02a},   $A+B$ is of type  $(FPV)$.
\end{proof}

We have been unable to relax the convexity hypothesis in Corollary \ref{domain:L1}.

We finish by listing some related interesting, at least to the current authors, questions regarding the sum problem.
\begin{problem}
 Let  $A:X\rightrightarrows X^*$ be maximally monotone with convex domain.
  Is
$A$ necessarily of type (FPV)?
\end{problem}

Let us recall a problem posed by S. Simons in \cite[Problem~41.2]{Si}

\begin{problem}
Let $A:X\rightrightarrows X^*$ be of type (FPV),
let $C$ be a nonempty closed convex subset of $X$,
and suppose that $\dom A \cap \inte C\neq \varnothing$.
Is $A+N_C$ necessarily maximally monotone$?$

\end{problem}

More generally, can we relax or indeed entirely drop the starshaped hypothesis on dom $A$ in Theorem \ref{TePGV:1}?

\begin{problem}
Let $A, B:X\To X^*$ be  maximally monotone
with $\dom A\cap\inte\dom B\neq\varnothing$. Assume
that $A$ is of type (FPV).
Is $A+B$ necessarily maximally monotone$?$

\end{problem}

If all maximally monotone operators are type (FPV) this is no easier than the full sum problem. Can the results of \cite{Ves} help here?

\section*{Acknowledgment}

Jonathan  Borwein and Liangjin Yao were partially supported
by various Australian Research Council grants.

\section{Appendix }\label{Appenc:1}

\textbf{Proof of Case 2 in the proof of Theorem~\ref{TePGV:1}.}

\begin{proof}
\emph{Case 2}: There exists $N\in\NN$ such that
 $\lambda_n a_n\not\in\dom B$ for every $n\geq N$.

We can and do suppose that $ \lambda_n a_n\not\in\dom B$ for every $n\in\NN$.
Thus $a_n\neq0$ for every $n\in\NN$.

 Since $0\in\inte\dom B$, \eqref{TCoSM:e4} and \eqref{TCoSM:e3}
imply that $0<\lambda_{\infty}$ and hence by \eqref{TCoSM:da1}
\begin{align}
0<\lambda_{\infty}<1.\label{TCoSM:d2}
\end{align}
By \eqref{TCoSM:d1}, $\|a_n-z\|\neq0$ for every $n\in\NN$.

Fix $n\in\NN$.
Since $0\in\inte\dom B$, there exists $0<\rho_0\leq1$ such that
$\rho_0 B_X\subseteq \dom B$.
Since $0\in\sta(\dom A)$ and $a_n\in\dim A$, $\lambda_n a_n\in\dom A$.
Set \begin{align}
b_n:=\lambda_na_n\quad\text{ and take} \quad b^*_n\in A(\lambda_na_n).
\label{TCoSM:d5}
\end{align}
Next we show that there exists $\varepsilon_n\in\left]0,\frac{1}{n}\right[$ such that
\begin{align}
H_n\subseteq\dom B\quad\text{and}\quad
\inf\big\|B\big(H_n\big)\big\|\geq n(1+\tau_0\|b_n^*\|),\quad\varepsilon_n\max\{\|a_n\|,1\}<\frac{1}{2}\|z-a_n\|\lambda_n.
\label{TCoSM:d6}\end{align}
where $H_n:=(1-\varepsilon_n)b_n+\varepsilon_n \rho_0U_X$ and $\tau_0:=\frac{1}{\lambda_n}\left[2\|z\|+2\|a_n\|+2
+(\|a_n\|+1)\frac{2\lambda_n\|z-a_n\|}{\rho_0}\right]$.

For every $\varepsilon\in \left]0,1\right[$, by \eqref{TCoSM:e3} and Fact~\ref{f:referee02c},
$(1-\varepsilon)b_n+\varepsilon \rho_0 B_X\subseteq \overline{\dom B}$.
By Fact~\ref{f:referee02c} again,
$(1-\varepsilon)b_n+\varepsilon \rho_0 U_X\subseteq \inte\overline{\dom B}=\inte\dom B$.

Now we show the second part of \eqref{TCoSM:d6}.  Let $k\in\NN$ and $(s_k)_{k\in\NN}$ be  a positive sequence such that  $s_k\longrightarrow 0$
when $k\longrightarrow \infty$. It suffices to show
\begin{align}
\lim_{k\rightarrow\infty}\inf\big\|B\big((1-s_k)b_n+s_k \rho_0U_X\big)\big\|=+\infty.
\label{TCoSM:d7}
\end{align}
Suppose to the contrary there exist a sequence $(c_k,c^*_k)_{k\in\NN}$ in $\gra B\cap \left[\big((1-s_k)b_n+s_k \rho_0U_X\big)\times X^*\right]$ and $L>0$ such that $\sup_{k\in\NN}\|c^*_k\|\leq L$. Then
$c_k\longrightarrow b_n=\lambda_n a_n$.
By the Banach-Alaoglu Theorem
(see \cite[Theorem~3.15]{Rudin}),
 there  exist a weak* convergent subnet, $(c^*_{\beta})_{\beta\in J}$ of
$(c^*_k)_{k\in\NN}$ such that $c^*_{\beta}\weakstarly c^*_{\infty}\in X^*$.
\cite[Corollary~4.1]{BY1} shows that $(\lambda_n a_n, c^*_{\infty})\in\gra B$,
which contradicts our assumption that $\lambda_na_n\notin \dom B$.

Hence \eqref{TCoSM:d7} holds and so does \eqref{TCoSM:d6}.

Set $t_n:=\frac{\varepsilon_n\rho_0}{2\lambda_n\|z-a_n\|}$ and thus $0<t_n<\frac{1}{4}$. Thus
\begin{align}
t_n \lambda_{n} z+(1-t_n)(1-\varepsilon_n)b_n\in H_n.\label{TCSM:h1}
\end{align}
Next we show there exists $(\widetilde{a_n}, \widetilde{a_n}^*)_{n\in\NN}$
in $\gra A\cap (H_n\times X^*)$ such that
\begin{align}
\big\langle z-\widetilde{a_n},\widetilde{a_n}^*\big\rangle\geq-\tau_0\|b_n^*\|.\label{TCoSM:d8}
\end{align}
We consider two subcases.

\emph{Subcase 2.1}:
$\big(t_n \lambda_{n} z+(1-t_n)(1-\varepsilon_n)b_n, (1+t_n)b^*_n\big)\in\gra A$.

Then set $(\widetilde{a_n}, \widetilde{a_n}^*):=\big(t_n \lambda_{n} z+(1-t_n)(1-\varepsilon_n)b_n, (1+t_n)b^*_n\big)$. Since $(0,0)\in\gra A$,
$\langle b_n, b^*_n\rangle\geq0$. Then we have
\begin{align}
&\Big\langle t_n \lambda_{n} z-\widetilde{a_n},\, \widetilde{a^*_n}\Big\rangle
=\Big\langle t_n \lambda_{n} z-t_n \lambda_{n} z-(1-t_n)(1-\varepsilon_n)b_n,\, (1+t_n)b^*_n\Big\rangle\nonumber\\
&=\Big\langle-(1-t_n)(1-\varepsilon_n)b_n, (1+t_n)b^*_n\Big\rangle=
-\Big\langle(1-t^2_n)(1-\varepsilon_n)b_n, b^*_n\Big\rangle\geq
-\Big\langle b_n, b^*_n\Big\rangle.\label{TCoSM:d9}
\end{align}
On the other hand, \eqref{TCoSM:d5} and the monotonicity of $A$ imply that
\begin{align*}
\Big\langle t_n \lambda_{n} z+(1-t_n)(1-\varepsilon_n)b_n-b_n, t_nb^*_n\Big\rangle
=\Big\langle t_n \lambda_{n} z+(1-t_n)(1-\varepsilon_n)b_n-b_n, (1+t_n)b^*_n-b^*_n\Big\rangle\geq0
\end{align*}
Thus
\begin{align}
\Big\langle t_n \lambda_{n} z-\left[1-(1-t_n)(1-\varepsilon_n)\right]b_n, b^*_n\Big\rangle
\geq0.\label{TCoSM:d10}
\end{align}
Since $1-(1-t_n)(1-\varepsilon_n)>0$ and $\langle b_n, b_n^*\rangle
=\langle b_n-0, b_n^*-0\rangle\geq0$,
\eqref{TCoSM:d10} implies that
$\langle t_n \lambda_{n} z,b^*_n\rangle\geq0$ and thus
\begin{align*}\langle z,b^*_n\rangle\geq0.
\end{align*}
Then by $\widetilde{a_n}^*=(1+t_n)b^*_n$ and $t_n\lambda_n\leq 1$,
\eqref{TCoSM:d9} implies that
\begin{align*}
\Big\langle z-\widetilde{a_n},\, \widetilde{a^*_n}\Big\rangle\geq
-\Big\langle b_n,\, b^*_n\Big\rangle\geq-\|b_n\|\cdot\|b^*_n\|\geq-\|a_n\|\cdot\|b^*_n\|\geq-\tau_0\|b_n^*\|.
\end{align*}
Hence \eqref{TCoSM:d8} holds.

\emph{Subcase 2.2}: $\big(t_n \lambda_{n} z+(1-t_n)(1-\varepsilon_n)b_n, (1+t_n)b^*_n\big)\notin\gra A$.

Since $0\in\sta(\dom A)$ and $a_n\in\dom A$, we have $
(1-\varepsilon_n)\lambda_na_n\in\dom A$, hence $\dom A\cap H_n\neq\varnothing$.
Since $
t_n \lambda_{n} z+(1-t_n)(1-\varepsilon_n)b_n\in H_n$ by \eqref{TCSM:h1}, $\big(t_n \lambda_{n} z+(1-t_n)(1-\varepsilon_n)b_n, (1+t_n)b^*_n\big)\notin\gra A$ and  $A$ is of type (FPV),
there exists $(\widetilde{a_n}, \widetilde{a_n}^*)\in\gra A$ such that
$\widetilde{a_n}\in H_n$ and
\allowdisplaybreaks
\begin{align*}
&\Big\langle t_n \lambda_{n} z+(1-t_n)(1-\varepsilon_n)b_n
-\widetilde{a_n},
\,\widetilde{a_n}^*-(1+t_n)b^*_n\Big\rangle>0\\
&\Rightarrow\Big\langle t_n \lambda_{n} z-\left[1-(1-t_n)(1-\varepsilon_n)\right]\widetilde{a_n}+(1-t_n)(1-\varepsilon_n)(b_n
-\widetilde{a_n}),
\,\widetilde{a_n}^*-b^*_n\Big\rangle\\
&\quad>\Big\langle t_n \lambda_{n} z+(1-t_n)(1-\varepsilon_n)b_n
-\widetilde{a_n},
\,t_nb^*_n\Big\rangle\geq \Big\langle t_n \lambda_{n} z
-\widetilde{a_n},
\,t_nb^*_n\Big\rangle\quad\text{(since $\langle b_n, b^*_n\rangle\geq0$)}\\
&\Rightarrow\Big\langle t_n \lambda_{n} z-\left[1-(1-t_n)(1-\varepsilon_n)\right]\widetilde{a_n},
\,\widetilde{a_n}^*-b^*_n\Big\rangle\\
&\quad>\Big\langle (1-t_n)(1-\varepsilon_n)(b_n
-\widetilde{a_n}),
\,b^*_n-\widetilde{a_n}^*\Big\rangle+\Big\langle t_n \lambda_{n} z
-\widetilde{a_n},
\,t_nb^*_n\Big\rangle\nonumber\\
&\Rightarrow\Big\langle t_n \lambda_{n} z-\left[1-(1-t_n)(1-\varepsilon_n)\right]\widetilde{a_n},
\,\widetilde{a_n}^*-b^*_n\Big\rangle>\Big\langle t_n \lambda_{n} z
-\widetilde{a_n},
\,t_nb^*_n\Big\rangle\nonumber\\
&\Rightarrow \Big\langle t_n \lambda_{n} z-\left[t_n+\varepsilon_n-t_n\varepsilon_n\right]\widetilde{a_n},
\,\widetilde{a_n}^*\Big\rangle>\Big\langle t_n \lambda_{n} z
-\widetilde{a_n},
\,t_nb^*_n\Big\rangle+\Big\langle t_n \lambda_{n} z-\left[t_n+\varepsilon_n-t_n\varepsilon_n\right]\widetilde{a_n},
\,b^*_n\Big\rangle.
\end{align*}
Since $\langle \widetilde{a_n}, \widetilde{a_n}^*\rangle=
\langle \widetilde{a_n}-0, \widetilde{a_n}^*-0\rangle\geq0$ and $
t_n+\varepsilon_n-t_n\varepsilon_n\geq t_n\geq t_n\lambda_n$,
$\Big\langle \left[t_n+\varepsilon_n-t_n\varepsilon_n\right]\widetilde{a_n},
\,\widetilde{a_n}^*\Big\rangle\geq t_n\lambda_n\Big\langle\widetilde{a_n},
\,\widetilde{a_n}^*\Big\rangle$.  Thus
\begin{align*}
&\Big\langle t_n \lambda_{n} z-t_n\lambda_n\widetilde{a_n},
\,\widetilde{a_n}^*\Big\rangle>\Big\langle t_n \lambda_{n} z
-\widetilde{a_n},
\,t_nb^*_n\Big\rangle+\Big\langle t_n \lambda_{n} z-\left[t_n+\varepsilon_n-t_n\varepsilon_n\right]\widetilde{a_n},
\,b^*_n\Big\rangle\nonumber\\
&\Rightarrow \Big\langle \frac{t_n \lambda_{n} z-t_n\lambda_n\widetilde{a_n}}{\lambda_nt_n},
\,\widetilde{a_n}^*\Big\rangle>\Big\langle t_n \lambda_{n} z
-\widetilde{a_n},
\,\frac{1}{\lambda_n}b^*_n\Big\rangle+\Big\langle \frac{t_n \lambda_{n} z-\left[t_n+\varepsilon_n-t_n\varepsilon_n\right]\widetilde{a_n}}{\lambda_nt_n},
\,b^*_n\Big\rangle\\
&\Rightarrow \Big\langle   z-\widetilde{a_n},
\,\widetilde{a_n}^*\Big\rangle>\Big\langle t_n \lambda_{n} z
-\widetilde{a_n},
\,\frac{1}{\lambda_n}b^*_n\Big\rangle+\Big\langle z-\left[1+\frac{\varepsilon_n}{t_n}
-\varepsilon_n\right]\frac{1}{\lambda_n}\widetilde{a_n},
\,b^*_n\Big\rangle\\
&\Rightarrow \Big\langle  z-\widetilde{a_n},
\,\widetilde{a_n}^*\Big\rangle>
-\frac{1}{\lambda_n}\|b^*_n\|\big(\|z\|+\|a_n\|+1\big)
-\|b^*_n\|\Big(\| z\|+\frac{1}{\lambda_n}(\|a_n\|+1)\big(1+\frac{2\lambda_n\|z-a_n\|}{\rho_0}
\Big)\nonumber\\
&\Rightarrow \Big\langle  z-\widetilde{a_n},
\,\widetilde{a_n}^*\Big\rangle>
-\|b^*_n\|\frac{1}{\lambda_n}\left[2\|z\|+2\|a_n\|+2
+(\|a_n\|+1)\frac{2\lambda_n\|z-a_n\|}{\rho_0}
\right]=-\tau_0\|b^*_n\|.
\end{align*}
Hence combining all the subcases, we have \eqref{TCoSM:d8} holds.

Since $\varepsilon_n<\frac{1}{n}$ and $\widetilde{a_n}\in H_n$,
\eqref{TCoSM:e4} shows that
\begin{align}
\widetilde{a_n}\longrightarrow \lambda_{\infty}z.\label{TCoSM:d11}
\end{align}

Take $w^*_n\in B(\widetilde{a_n})$ by \eqref{TCoSM:d6}. Then by \eqref{TCoSM:d6} again,
\begin{align}
\|w^*_n\|\geq n(1+\tau_0\|b^*_n\|),\quad\forall n\in\NN.\label{TCoSM:d12}
\end{align}
Then by \eqref{TCoSM:d8}, we have
\begin{align*}
-\tau_0\|b^*_n\|+\Big\langle z-\widetilde{a_n}, w^*_n\Big\rangle+\Big\langle z^*, \widetilde{a_n}\Big\rangle&\leq
\Big\langle z-\widetilde{a_n}, \widetilde{a_n}^*\Big\rangle+
\Big\langle z-\widetilde{a_n}, w^*_n\Big\rangle+\Big\langle z^*, \widetilde{a_n}\Big\rangle\\
&\leq F_{A+B}(z,z^*)
\end{align*}
Thus
\begin{align}
-\frac{\tau_0\|b^*_n\|}{\|w^*_n\|}+\Big\langle z-\widetilde{a_n}, \frac{w^*_n}{\|w^*_n\|}\Big\rangle+\Big\langle \frac{z^*}{\|w^*_n\|}, \widetilde{a_n}\Big\rangle
&\leq \frac{F_{A+B}(z,z^*)}{\|w^*_n\|}.\label{TCoSM:d14}
\end{align}
By the Banach-Alaoglu Theorem
(see \cite[Theorem~3.15]{Rudin}),
 there  exist a weak* convergent subnet, $(\frac{w^*_i}{\|w^*_i\|})_{i\in I}$ of
$\frac{w^*_n}{\|w^*_n\|}$ such that
\begin{align}
\frac{w^*_i}{\|w^*_i\|}\weakstarly w^*_{\infty}\in X^*.\label{TCoSM:d18}
\end{align}
Combining \eqref{TCoSM:d11}, \eqref{TCoSM:d12} and \eqref{TCoSM:d18}, by $F_{A+B}(z,z^*)<+\infty$, take the limit along the subnet in \eqref{TCoSM:d14} to obtain
\begin{align*}
\Big\langle z-\lambda_{\infty} z, w^*_{\infty}\Big\rangle\leq0.
\end{align*}
Then  \eqref{TCoSM:d2} shows that
\begin{align}
\Big\langle z, w^*_{\infty}\Big\rangle\leq0.\label{TCoSM:d17}
\end{align}
On the other hand, since $0\in\inte\dom B$, Fact~\ref{extlem} implies that there exists
$\rho_1>0$ and $M>0$ such that
\begin{align*}
\Big\langle \widetilde{a_n}, w^*_n\Big\rangle\geq\rho_1\|w^*_n\|-(\|\widetilde{a_n}\|
+\rho_1)M.
\end{align*}
Thus
\begin{align*}
\Big\langle \widetilde{a_n}, \frac{w^*_n}{\|w^*_n\|}\Big\rangle\geq\rho_1-\frac{(\|\widetilde{a_n}\|
+\rho_1)M}{\|w^*_n\|}.
\end{align*}
Combining \eqref{TCoSM:d11}, \eqref{TCoSM:d12} and \eqref{TCoSM:d18},
 take the limit along the subnet in the above inequality to obtain
\begin{align*}
\Big\langle \lambda_{\infty}z, w^*_{\infty}\Big\rangle\geq\rho_1.
\end{align*}
Hence
\begin{align*}
\Big\langle z, w^*_{\infty}\Big\rangle\geq\frac{\rho_1}{\lambda_{\infty}}>0,
\end{align*}
which contradicts \eqref{TCoSM:d17}.

\end{proof}

\end{document}